\documentclass[11pt]{amsart} 
\usepackage{amssymb}
\usepackage{amsmath}
\usepackage{setspace}
\usepackage{amsthm}
\usepackage{amsfonts}
\usepackage{amstext}
\usepackage{mathtext}
\usepackage{hyperref}

\newtheorem{theorem}{Theorem}

\newtheorem{proposition}[theorem]{Proposition}

\newtheorem{remark}[theorem]{Remark}

\newtheorem{definition}[theorem]{Definition}

\begin{document}

\title{SOME PROBLEMS IN ADDITIVE NUMBER THEORY (v.3)} 

\author{Andrei Allakhverdov}
\address{46 Sedykh st, flat 5, 220103, Minsk, Belarus}
\email{andall1952@gmail.com}

\begin{abstract}
In this article we present method of solving some additive problems with primes.
The method may be employed to the Goldbach-Euler conjecture and the twin primes conjecture.
The presented method  also makes it possible to obtain some interesting results related to the densities of sequences. 
The method is based on the direct construction of the Eratosthenes-type double sieve and does not use empirical and heuristic reasoning.
\end{abstract}

 

\date{date}

\thanks{}

\maketitle


\hfill \begin{minipage}[h]{0.75\textwidth}
All results of the profoundest mathematical investigation must ultimately be expressible in the simple form of properties of integers.

 \begin{flushright}
Leopold Kronecker
\end{flushright}

Regarding the relative powers of elementary sieve methods and the analytical methods one usually considers that the latter should be more powerful... 
But history has shown that such views are not totally correct.

 \begin{flushright}
H.E. Richert, Lectures on Sieve Methods, \\
Tata Institute of Fundamental Research, 1976
\end{flushright}
 \end{minipage}

\section{Background and conventions} 
The method is based on some properties of $\mathbf{Z}/6\mathbf{Z}$, namely 
the residue classes $\bar{1}_6$ and $\bar{5}_6$ contain all odd primes except prime number $3$, 
each of the even residue classes modulo $6$ may be represented as a sum and a difference of $\bar{1}_6$ and/or $\bar{5}_6$, 
the sequences $\bar{1}_6$ and $\bar{5}_6$ are well-structured by all prime numbers. 
It allows us to construct the double sieve.

Let ${P}$ denote the set of all primes. We assume that $p \in \mathcal{P}$, where $\mathcal{P} = {P} \setminus \{ 2, 3\}$.

We will use the notation $\#^* S_m$ for the number of nonzero terms of the segment of sequence $S_m$.

\section{Preliminaries}

\subsection{Well-structured sequences}\label{primes_and_composite} 
\begin{definition} 
{We say that a sequence $S$ is well-structured by number $q$ if the indices of elements of sequence $S$ that are divisible by $q$ form an arithmetic progression with the common difference $q$.} 
\end{definition}
Let $ A = \{ \, a_i \colon a_i = 6i-1, i \in \mathbf{N} \} $, and let $ B = \{ \, b_i \colon b_i = 6i+1, i \in \mathbf{N}  \} $:
\tabcolsep=0.2em
 \begin{center}
\begin{tabular} {rrrrrrrrrrrrrrrrrrrrrr}
{${A}$  = } \{ & {5,} &{11,} &{17,} &{23,} &{29,} &{35,} & {41,} & {47,} & {53,} & {59,} & {65,} & {71,} & {77,} & {83,}  & {89,} & {95,} &  {101,} &  \ldots \}, \\
{${B}$  = } \{ & {7,} &{13,} &{19,} &{25,} &{31,} &{37,} & {43,} & {49,} & {55,} & {61,} & {67,} & {73,} & {79,} &  {85,}  & {91,} & {97,} & {103,} & \ldots \}.     
\end{tabular}
\end{center} 
The terms of sequences $A$ and $B$ either are primes $p \in \mathcal{P}$ or products of primes $p \in \mathcal{P}$. 
Obviously $A \cup B \supset  \mathcal P$.
\begin{theorem}  
The infinite sequences $A$ and $B$ are well-structured over all primes $p \in \mathcal{P}$. 
\end{theorem}

\begin{proof}
If $a_i$ is composite, then there exist $j \not=0$, $ k \not= 0$ such that   
\begin{align*} 
 { a_i=6i - 1 = a_j  b_k = (6j - 1)(6k + 1) = 36jk - 6k + 6j - 1}. 
\end{align*}
In this case, we have two expressions for $i$:
\begin{align*}
    i &= k\,(6j - 1) + j = k a_j +j, 	 \\	
    i &= j\,(6k + 1) - k = jb_k -k,       
  \end{align*}
	that determine two families of arithmetic progressions 
\begin{align} 
	a_j \mid a_i & \Leftrightarrow i = + j + k a_j, \label{a_a_composite} \\ 
	b_k \mid a_i & \Leftrightarrow i = - k + j b_k, \label{a_b_composite}
	\end{align}	

If $b_i$ are composite, then there exist $j\not=0$, $ j'\not=0,$ or $ k\not=0, k'\not=0$ such that at least one of the following equalities will hold:
\begin{align*}
  b_i &= 6i + 1 = a_j a_{j'} = \left( {6j - 1} \right)\left( {6j'- 1} \right)  = 6\left( {6 j j' - j - j' } \right) + 1, \\  
  b_i &= 6i + 1 = b_k b_{k'} = \left( {6k + 1} \right)\left( {6k'+ 1} \right)  = 6\left( {6 k k' + k + k'} \right) + 1.   
\end{align*} 
In these cases we also have two expressions for $i$:
\begin{align*}
   i &= j' ( 6j - 1 ) - j = j' a_j - j, 	 \\  
   i &= k' ( 6k + 1 ) + k = k'b_k + k,          	
\end{align*}
that determine two families of arithmetic progressions 	
	\begin{align}  
	a_j \mid b_i &\Leftrightarrow i = - j + j' a_j ,  \label{b_a_composite}	\\
	b_k \mid b_i &\Leftrightarrow i = + k + k' b_k.   \label{b_b_composite}	          
	\end{align}
Expressions (\ref{a_a_composite}), (\ref{a_b_composite}), (\ref{b_a_composite}), and (\ref{b_b_composite}) implies that the sequences $A$ and $B$ are well-structured by all numbers $a_j \in A$ and $b_k \in B$. 
Since $A \cup B \supset  \mathcal P$ the proof is complete. 
\end{proof}		

\subsection{Sieving of sequences} \label{sieve_method}
That is, for any $p \in \mathcal{P}$ in each of the sequences ${A}$ and ${B}$ there is exactly one infinite subsequence of terms, which are divisible by $p$ and form an arithmetic progression with the common difference $6p$. 
The indices of these terms also form an arithmetic progression with the common difference $p$. 

We define the sieving of a sequence $S$ by a number $p$ as replacing composite terms of sequence $S$ that are a multiple of $p$ by $0$. 
We denote this infinite subsequence of zeros by $\{0\}_p$.
A sequence $S$ sifted by $p$ we denote by $ S {\setminus \lambda p}$.
A sequence $S$ sifted over all $p \in \mathcal P$ we denote by $S {\setminus \lambda \mathcal P}$. 
We write $S_m {\setminus \lambda p}$ and $S_m {\setminus \lambda \mathcal P}$ for sieving of a segment $S_m$.  

\subsection{Sequences $\mathcal A$, $\mathcal B$, $\mathcal L$, and $\mathcal R$} \label{sequences L and R}

Let $\mathcal A = A \setminus \lambda \mathcal P$ and let $\mathcal B = B \setminus \lambda \mathcal P$.
The rules of correspondence between sequences $\mathcal L, \mathcal R$ on one side and the sequences $A, B$ on the other side are 
$ {\mathcal L} = \left\{  l_i \colon {l_i} = i \text{ if } {a_i} \in \mathcal{P}; \right.$ $\left. {l_i} = 0  \text{ if } {a_i} \notin \mathcal{P} \right\}$ and
$ {\mathcal R} = \left\{  r_i \colon {r_i} = i \text{  if  } {b_i} \in \mathcal{P};  {r_i} = 0  \text{  if  } {b_i} \notin \mathcal{P} \right\},$
i.e.  
\tabcolsep=0.15em
 \begin{center}
\begin{tabular} {rrrrrrrrrrrrrrrrrrrrrr}
{${\mathcal A}$  = } \{ &{5,} &{11,} &{17,} &{23,} &{29,} &{0,} &{41,} &{47,} &{53,} &{59,} &{0,}  &{71,} &{0,}  &{83,} &{89,} &{0,}  &{101,} &  \ldots \}, \\
\vspace{10pt}
{${\mathcal L}$  = } \{ &{1,} &{2,} &{3,} &{4,} &{5,} &{0,} &{7,} &{8,} &{9,} &{10,} &{0,}  &{12,} &{0,}  &{14,} &{15,} &{0,}  &{17,} &  \ldots \}, \\
{${\mathcal B}$  = } \{ &{7,} &{13,} &{19,} &{0,} &{31,} &{37,} &{43,} &{0,} &{0,} &{61,} &{67,} &{73,} &{79,} &{0,}  & {0,} &{97,} &{103,} & \ldots  \}, \\
{${\mathcal R}$  = } \{ &{1,} &{2,} &{3,} &{0,} &{5,} &{6,} &{7,} &{0,} &{0,} &{10,} &{11,} &{12,} &{13,} &{0,}  & {0,} &{16,} &{17,} & \ldots  \}.
\end{tabular}     
\end{center}
All nonzero terms of the sequences ${\mathcal A}$ and ${\mathcal B}$ are prime numbers.
All nonzero terms of the sequences ${\mathcal L}$ and ${\mathcal R}$ are indices of appropriate nonzero terms of the sequences ${\mathcal A}$ and ${\mathcal B}$, respectively.
Obviously, the sequences $\mathcal L$ and $\mathcal R$ inherit the structures of the sequences $\mathcal A$ and $\mathcal B$ in the sense of distribution of zero terms.  
The indices of zero terms of the sequences ${\mathcal A}$ and ${\mathcal L}$ are determined by the right-hand sides of (\ref{a_a_composite}) and 
(\ref{a_b_composite}); the indices of zero terms of the sequences ${\mathcal B}$ and ${\mathcal R}$ are determined by the right-hand sides of 
(\ref{b_a_composite}) and (\ref{b_b_composite}). 

Now we define sequences 
${\mathcal A}^{ m'} = \{  a_i^{ m'} \colon a_i^{ m'} =  a_{i + m'} \}$ \text{and} 
${\mathcal B}^{ m'} = \{  b_i^ { m'} \colon b_i^{ m'} =  b_{i + m'} \}$, 
as the remainder of sequences ${\mathcal A}$ and ${\mathcal B}$ after the $m'$-th  term, for example, 
\tabcolsep=0.15em
 \begin{center}
\begin{tabular} {rrrrrrrrrrrrrrrrrrrrrr}
{${\mathcal A}^{\, 5}$ = } \{&{0,}&{41,}&{47,}&{53,}&{59,} &{0,} &{71,} &{0,} &{83,} &{89,} &{0,} &{101,}&{107,}&{113,} &{0,} &{0,}&{131,} &  \ldots \}, \\
{${\mathcal B}^{\, 6}$ = } \{&{43,} &{0,}&{0,}&{61,}&{67,} &{73,} &{79,} &{0,} &{0,} &{97,} &{103,} &{109,} &{0,} &{0,} &{127,} &{0,} &{139,} & \ldots \}.
\end{tabular}
\end{center} 
Just the same for sequences ${\mathcal L}^{ m'}$ and ${\mathcal R}^{ m'}$ we get
\tabcolsep=0.15em
 \begin{center}
\begin{tabular} {rrrrrrrrrrrrrrrrrrrrrr}
{${\mathcal L}^{\, 5}$ = } \{&{0,} &{7,} &{8,} &{9,}&{10,} &{0,} &{12,} &{0,} &{14,} &{15,} &{0,}  &{17,} &{18,} &{19,} &{0,} &{0,} &{22,} & \ldots \}, \\
{${\mathcal R}^{\, 6}$ = } \{&{7,} &{0,} &{0,} &{10,} &{11,} &{12,} &{13,} &{0,} & {0,} &{16,} &{17,} &{18,} &{0,} &{0,} &{21,} &{0,} &{23,} & \ldots \}.  
\end{tabular}
\end{center} 
By $S\{ g \}$ we denote a sequence, the terms of which someway determined.

\subsection{Direct and inverse segments of sequences}\label{seg_of_seq}

Let $S_{m} = \{ s_i \}_{i=1}^{i=m}$ denotes the initial segment of length $m$ of sequence $S$, for example, 
\tabcolsep=0.3em
 \begin{center}
\begin{tabular} {rrrrrrrrrrrrrrrr}
{${\mathcal A}_{14}$ } &= ({5,} & {11,} & {17,} & {23,} & {29,} & {0,} & {41,} & {47,} & {53,} & {59,} & { 0,} & {71,} & { 0,} & {83}),   \\
{${\mathcal B}_{14}$ } &= ({7,} & {13,} & {19,} & { 0,} & {31,} & {37,} & {43,} & {0,} & {0,} & {61,} & {67,} & {73,} & {79,} & { 0}).   
\end{tabular}
\end{center} 
We call these segments \textit{the direct segments}, and we call the segments 
\tabcolsep=0.295em
 \begin{center}
\begin{tabular}{rrrrrrrrrrrrrrrr}
{${\mathcal A'}_{14}$  } &= ({83,} & { 0,} & {71,} & { 0,} & {59,} & {53,} & {47,} &{41,} &{0,} &{29,} &{23,} &{17,} &{11,} & {5}), \\
{${\mathcal B'}_{14}$  } &= ({ 0,} & {79,} & {73,} & {67,} & {61,} & {0,} & {0,} &{43,} &{37,} &{31,} &{0,} &{19,} &{13,} & {7}),  \\ 
\end{tabular}
\end{center} 
\textit{the inverse segments}.
Just the same for sequences ${\mathcal L}_{m}$, ${\mathcal R}_{m}$, $\mathcal L'_{m} $ and $\mathcal R'_{m}$:
\tabcolsep=0.3em
 \begin{center}
\begin{tabular}{lrrrrrrrrrrrrrrrr}
{${\mathcal L}_{14} $} &= ({ 1,} & {2,} & {3,} & {4,} & {5,} & {0,} & {7,} & {8,} & {9,} & {10,} & { 0,} & {12,} & { 0,} & {14}),   \\
{${\mathcal R}_{14} $} &= ({ 1,} & {2,} & {3,} & {0,} & {5,} & {6,} & {7,} & {0,} & {0,} & {10,} & {11,} & {12,} & {13,} & { 0}), \\
{${\mathcal L'}_{14}$} &= ({14,} & { 0,} & {12,} & { 0,} & {10,} & {9,} & {8,} &{7,} &{0,} &{5,} &{4,} &{3,} &{2,} & {1}), \\
{${\mathcal R'}_{14}$} &= ({ 0,} & {13,} & {12,} & {11,} & {10,} & {0,} & {0,} &{7,} &{6,} &{5,} &{0,} &{3,} &{2,} & {1}). 
\end{tabular}
\end{center}

\subsection{}

Let $\pi (a,n)$ denote the number of primes not exceeding $n$ that are of the form $6i-1$, and let $\pi (b,n)$ denote the number of primes not exceeding $n$ that are of the form $6i+1$. 
For the following discussion, we take $n=6m$. Then, 
$$\begin{array}{*{20}l}
   {\pi }\left(a, n \right) &= {\pi }\left(a, 6m \right) &= \#^* {\mathcal A}_{m} &= \#^* {\mathcal L}_{m} &= \sum\nolimits_{l_i  \leqslant m, \, l_i  \ne 0} 1 , 	\\
	 {\pi }\left(b, n \right) &= {\pi }\left(b, 6m \right) &= \#^* {\mathcal B}_{m} &= \#^* {\mathcal R}_{m} &= \sum\nolimits_{r_i  \leqslant m, \, r_i  \ne 0} 1 .  
 \end{array} $$

\subsection{Even numbers}

Let ${\mathcal G} = \{\mathrm g \}$ be the set of all positive even numbers. 
We partition set ${\mathcal G} \setminus \{2\}$ into three disjoint sequences (residue classes modulo $6$) and we will assume that   
${\mathcal G^1} = \{ \mathrm g_m^1 : \mathrm g_m^1 = 6m - 2 \}$, 
${\mathcal G^2}=  \{ \mathrm g_m^2 : \mathrm g_m^2 = 6m \}$, 
${\mathcal G^3}=  \{ \mathrm g_m^3 : \mathrm g_m^3 = 6m + 2\}$.  
Each integer $m$ determines three consecutive even numbers, one from each of these classes.

\subsection{Summation of Sequences}\label{sum_of_segments}

We define the subtraction of two sequences $S' = \{s'_i\}$ and $S'' = \{s''_i\}$ as a sequence
$S = \{s_i: s_i = s'_i - s''_i  \text{  if  } s'_i s''_i \ne 0; s_i = 0  \text{  if  } s'_i s''_i = 0 \}$. 

We define the addition of segments of two sequences $S'_m = \{s'_i\}_{i=1}^{i=m}$ and $S''_m = \{s''_i\}_{i=1}^{i=m}$ 
as the segment of sequence $S_m = \{s_i: s_i = s'_i + s''_i  \text{  if  } s'_i s''_i \ne 0; s_i = 0  \text{  if  } s'_i s''_i = 0 \}$.  

We define the addition of $k>2$ sequences $S^1 = \{s^1_i\}, S^2 = \{s^2_i\}, \ldots S^k = \{s^k_i\}$ as sequence
$S = \{s_i: s_i = 1  \text{ if } s^1_i \cdot s^2_i \cdot \ldots \cdot s^k_i \ne 0; s_i =0  \text{ if } s^1_i \cdot s^2_i \cdot \ldots \cdot s^k_i = 0 \}$.

We assume that a sequence $S = S' + S''$ sieved by $p$ if both sequences $S'$ and $S''$ sieved by $p$, 
that is $S \setminus \lambda p = S' \setminus \lambda p + S'' \setminus \lambda p$. 

\subsection{Binary additive problems} 

\subsubsection{Pairs of primes with a fixed difference} \label{primes_with_fixed_gap}

Let ${\pi}_{{\mathrm g}}(n)$ be a number of primes $p$ not exceeding $n$ and such that  $p' = p + \mathrm g$ are also prime.  
All even numbers of each class may be represented as a difference of two odd integers from $A$ and/or $B$ in the only way, that is, 
$\mathrm g_{m'}^1=a_{i+m'}-b_i$, $\mathrm g_{m'}^2=(a_{i+m'}-a_i)=(b_{i+m'}-b_i)$, and $\mathrm g_{m'}^3=b_{i+m'}-a_i$. 
These identities allow us to find the solution of this problem from the constructions
\begin{align}
{\pi}_{{\mathrm g}^1_{m'}} (n+1) &= \#^* \left( {\mathcal A}^{m'}_{m}-{\mathcal B}_{m} \right) \label{a-b} \\
                                 &= \#^* \left( {\mathcal L}^{m'}_{m}-{\mathcal R}_{m} \right),  \nonumber\\   
{\pi}_{{\mathrm g}^2_{m'}} (n+1) &= \#^* \left( {\mathcal A}^{m'}_{m}-{\mathcal A}_{m} \right) + \#^* \left( {\mathcal B}^{m'}_{m}-{\mathcal B}_{m} \right) \label{a-a}\\  
	                               &= \#^* \left( {\mathcal L}^{m'}_{m}-{\mathcal L}_{m} \right) + \#^* \left( {\mathcal R}^{m'}_{m} - {\mathcal R}_{m} \right), \nonumber\\ 
  {\pi}_{{\mathrm g}^3_{m'}} (n) &= \#^* \left( {\mathcal B}^{m'}_{m}-{\mathcal A}_{m} \right) \label{b-a}  \\
	                               &= \#^* \left( {\mathcal R}^{m'}_{m}-{\mathcal L}_{m} \right). \nonumber
  \end{align}
We note that ${\pi}_{{\mathrm g}^1_{m'}} (n+1) - {\pi}_{{\mathrm g}^1_{m'}} (n) \leqslant 1$ and ${\pi}_{{\mathrm g}^2_{m'}} (n+1) - {\pi}_{{\mathrm g}^2_{m'}} (n) \leqslant 1$.

\begin{remark} \label{rem_on_diff}
If $3 \mid \mathrm g$ then ${\pi}_{\mathrm g}(n)$ is equal to the sum of the numbers of nonzero terms in two segments.
\end{remark}

For example, we will find the number of pairs $p,  p + 28$ where $p \leqslant 126 + 1 $.  
We have $28 = \mathrm g_5^1 \in \mathcal G^1$, $m' = 5$, and $m = 126/6 = 21$. Now applying construction (\ref{a-b}), we get  
\tabcolsep=0.085em
\begin{center}
\begin{tabular}{rrrrrrrrrrrrrrrrrrrrrrrrrrr}
 ${\mathcal A}^5_{21} =$ &{0,}&{41,}&{47,}&{53,}&{59,}&{0,}&{71,}&{0,}&{83,}&{89,}&{0,}&{101,}&{107,}&{113,}&{0,}&{0,}&{131,}&{137,}&{0,}&{149,}&{0} \\ 
   ${\mathcal B}_{21} =$ & {7,} &{13,}&{19,}&{0,}&{31,}&{37,}&{43,}&{0,}&{0,}&{61,}&{67,}&{73,}&{79,}&{0,}&{0,}&{97,}&{103,}&{109,}&{0,}& {0,}&{127} \\  
 \hline 
$(\mathcal A^5_{21} - \mathcal B_{21})  =$ &{0,}&{28,}&{28,}&{0,}&{28,}&{0,}&{28,}&{0,}&{0,}&{28,}&{0,}&{28,}&{28,}&{0,}&{0,}&{0,}&{28,}&{28,}&{0,}
                                               &{0,}&{0}.                                                        
\end{tabular}
\end{center}
\vspace{2pt}
Each nonzero term in segment $(\mathcal A^5_{21} - \mathcal B_{21}) $ indicates to one representation of number $\mathrm g_5^1 = 28$ as a difference of two primes.
Thus, $\pi_{28} (127) = \#^* ( (\mathcal A^5_{21} - \mathcal B_{21})  ) = 9$. 

Substitution $m' = 5$ and $m = 126/6 = 21$ in constructions (\ref{a-a}) and (\ref{b-a}) gives us the number of pairs of primes with difference equal to $30$ and $32$ respectively.

\subsubsection{Twin primes} \label{twinprimes} 

The twin primes are an important special case of pairs of primes with a fixed difference. 
We have the identity $ b_{i}-a_i = 2 $ for twin primes, and the construction 
$${\pi}_2 (n)  = \#^* \left( {\mathcal B}_{m}  -  {\mathcal A}_{m} \right) = \# \left( {\mathcal R}_{m}  -  {\mathcal L}_{m} \right)$$ 
for the number of pairs of twin primes $\pi_2 (n)$ not exceeding $n = 6m$.
Let $B - A = T$.
Since ${\mathcal A} = A {\setminus \lambda \mathcal P}$, ${\mathcal B} = B {\setminus \lambda \mathcal P}$ it follows
\tabcolsep=0.092em
\begin{center}
\begin{tabular}{rrrrrrrrrrrrrrrrrrrrrrrrrrr}
${{\mathcal B}}$ = &{7,}&{13,}&{19,}&{0,} &{31,}&{37,}&{43,}&{0,} &{0,} &{61,}&{67,}&{73,}&{79,}&{0,} &{0,} &{97,}&{103,}&{109,}&{0,}  &{0,}&{127, \ldots} \\ 
${{\mathcal A}}$ = &{5,}&{11,}&{17,}&{23,}&{29,}&{0,} &{41,}&{47,}&{53,}&{59,}&{0,} &{71,}&{0,} &{83,}&{89,}&{0,} &{101,}&{107,}&{113,}&{0,}&{0, \ldots}   \\ 
\hline
${{T} \setminus \lambda \mathcal P}$ = &{2,}&{2,}&{2,}&{0,}&{2,}&{0,}&{2,}&{0,}&{0,}&{2,}&{0,}&{2,}&{0,}&{0,}&{0,}&{0,}&{2,}&{2,} &{0,}  &{0,}&{0, \ldots}   
\end{tabular}
\end{center}
where each term equals to $2$ in sequence ${{T} \setminus \lambda \mathcal P}$ indicates at one pair of twin primes.
Now we consider very important sequence
$ {\mathcal T} = \left\{  t_i \colon {t_i} = i \text{  if  } {l_i}  {r_i} \ne 0 ; {t_i} = 0  \text{  if  } {l_i}  {r_i} = 0 \right\}$: 
\tabcolsep=0.094em
\begin{center}
\begin{tabular}{rrrrrrrrrrrrrrrrrrrrrrrrrrr}
${{\mathcal R}}= $ &{1,}&{2,}&{3,}&{0,}&{5,}&{6,}&{7,}&{0,}&{0,}&{10,}&{11,}&{12,}&{13,}&{0,}&{0,}&{16,}&{17,}&{18,}&{0,}&{0,}&{21,}&{0,}&{23,}&{0,}&{25 \ldots} \\ 
${{\mathcal L}} =$ &{1,}&{2,}&{3,}&{4,}&{5,}&{0,}&{7,}&{8,}&{9,}&{10,}&{0,}&{12,}&{0,}&{14,}&{15,}&{0,}&{17,}&{18,}&{19,}&{0,}&{0,}&{22,}&{23,}&{0,}&{25 \ldots} \\
\hline
${{\mathcal T}}= $ &{1,}&{2,}&{3,}&{0,}&{5,}&{0,}&{7,}&{0,}&{0,}&{10,}&{0,}&{12,}&{0,}&{0,}&{0,}&{0,}&{17,}&{18,}&{0,}&{0,}&{0,}&{0,}&{23,}&{0,}&{25 \ldots}  \\
\end{tabular}
\end{center}
If $t_i \ne 0$ then $6t_i \mp 1$ are both primes, i.e. form a pair of twin primes. 
It is easy to see that the sequence $\mathcal T$ inherits the structure of the sequence $T \setminus \lambda \mathcal P$ in the sense of distribution of zero terms.

\subsubsection{Representation of an even numbers as the sum of two primes } \label{reprofevennumber}

Denote by $G(\mathrm g ; p)$ the number of representations number $\mathrm g$ as a sum of two primes $p \in \mathcal P$.
All even numbers of each class may also be represented as a sum of two odd integers from $A$ and/or $B$ in the only way, 
$\mathrm g_{m+1}^1=a_i+a_{m-i+1}$, $\mathrm g_{m+1}^2=a_i+b_{m-i+1}$, and $\mathrm g_{m+1}^3=b_j+b_{m-j+1}$. 
These identities allow us to find the solution of this problem for all even numbers $\mathrm g \geqslant 10$ from the constructions
\begin{align}
G(\mathrm g^1_{m}; p )  & \geqslant 0.5 \cdot \#^* \left( {\mathcal A}_{m-1} + {\mathcal A'}_{m-1} \right) \label{a+a} \\
                         &= 0.5 \cdot \#^* \left( {\mathcal L}_{m-1} + {\mathcal L'}_{m-1} \right), \nonumber \\ 
G(\mathrm g^2_{m}; p )  & = 1.0 \cdot \#^* \left( {\mathcal A}_{m-1}+{\mathcal B'}_{m-1} \right) \label{a+b} \\
                         &= 1.0 \cdot \#^* \left({\mathcal L}_{m-1} +{\mathcal R'}_{m-1} \right), \nonumber \\ 
G(\mathrm g^3_{m}; p )  & \geqslant 0.5 \cdot \#^* \left( {\mathcal B}_{m-1} + {\mathcal B'}_{m-1} \right) \label{b+b} \\
		                     &= 0.5 \cdot \#^* \left( {\mathcal R}_{m-1} + {\mathcal R'}_{m-1} \right).  \nonumber   
\end{align}
\begin{remark} \label{rem_on_sum}
We take the coefficient $0.5$ in (\ref{a+a}) and (\ref{b+b}) because the sum of direct segment and reverse segment of the same class is the symmetrical segment where elements located symmetrically with respect to its center differ only by order of summation.
\end{remark}

For example, we will find the number of representations of even number $94$ as the sum of two primes.
We have $94 = \mathrm g_{16}^1 \in \mathcal G^1$, $m = 16$. Applying construction (\ref{a+a}) we obtain
\tabcolsep=0.32em
\begin{center}
\begin{tabular}{rrrrrrrrrrrrrrrrr}
  ${\mathcal A}_{15}  $  =  & {5,}  & {11,} & {17,} & {23,} & {29,} &{0,}  &{41,} &{47,} &{53,} &{59,} &{0,}  &{71,} &{0,}  &{83,}  &{89.}   \\
  ${\mathcal A'}_{15} $  =  & {89,} &{83,}  & {0,}  & {71,} & {0,}  &{59,} &{53,} &{47,} &{41,} &{0,}  &{29,} &{23,} &{17,} &{11,}  &{5.}   \\
\hline
  $(\mathcal A_{15} + \mathcal A'_{15} )$  = &{94,} &{94,} &{0,} &{94,}  &{0,}  &{0,}  &{94,} &{94,} &{94,}  &{0,} &{0,}  &{94,} &{0,}  &{94,} &{94.}					
\end{tabular}
\end{center}	
Here $0.5 \cdot \#^* (\mathcal A_{15} + \mathcal A'_{15} ))  = 4.5$ while $G(94; p) = 5 $. 
This result can be obtained by using last expression of (\ref{a+a}): 
\tabcolsep=0.38em
\begin{center}
\begin{tabular}{rrrrrrrrrrrrrrrrr}
     ${\mathcal L}_{15}  $ =  &{1,}  &{2,}  &{3,} &{4,}  &{5,} &{0,}  &{7,} &{8,}  &{9,}  &{10,} &{0,} &{12,} &{0,} &{14,} &{15.} \\
     ${\mathcal L'}_{15} $ =  &{15,} &{14,} &{0,} &{12,} &{0,} &{10,} &{9,} &{8,}  &{7,}  &{0,}  &{5,} &{4,}  &{3,} &{2,}  &{1.}   \\
\hline
  $({\mathcal L}_{15} + {\mathcal L'}_{15}) $  =&{16,} &{16,} &{0,} &{16,} &{0,} &{0,} &{16,} &{16,} &{16,} &{0,}  &{0,} &{16,} &{0,} &{16,} &{16.}
\end{tabular}
\end{center}
Constructions (\ref{a+b}) and (\ref{b+b}) with $m = 16$ give $G(96; p)$ and $G(98; p)$ respectively. 
\begin{remark}\label{rem_on_density}
If each of three consecutive even numbers $\mathrm g_{m}^1$, $\mathrm g_{m}^2$, and $\mathrm g_{m}^3$ may be represented as the sum of two primes then number $m$ has all three representations $m=l+l$, $m=l+r$, and $m=r+r$.
\end{remark}

\subsection{Double sieve}\label{double_sieving}  

In sections \ref{primes_and_composite}, \ref{sieve_method}, and \ref{sequences L and R} we showed that the sequences $A$ and $B$ are well-structured by all primes $p \in \mathcal{P}$.
Thus, for every $p \in \mathcal{P}$ in each of the sequences ${\mathcal A}$, ${\mathcal B}$, ${\mathcal L}$, and ${\mathcal R}$ there exists one and only one infinite subsequence $\{0\}_p$.
For every $p \in \mathcal{P}$ we have
\begin{align*}
\#^* (A_{m} {\setminus \lambda p}) = &\left| \left\{ a_i \colon i \leqslant m, \left( p, a_i \right) = 1  \right\}  \right| \sim m \left( 1 - 1/p \right), \\
\#^* (B_{m} {\setminus \lambda p}) = &\left| \left\{ b_i \colon i \leqslant m, \left( p, b_i \right) = 1  \right\}  \right| \sim m \left( 1 - 1/p \right).
\end{align*}
\begin{definition} 
We say that infinite sequence $S$ double sifted by prime $p$ if there exist two various infinite disjoint subsequences $\{0\}_p$.
\end{definition} 
\begin{definition} 
We say that sequence $S$ is a realization of double sieve if sequence $S$ double sifted over all $p \in \mathcal P$. 
\end{definition} 
\begin{definition} 
We say that sequence $S$ is a realization of double sieve if sequence $S$ double sifted over all but finitely many $p \in \mathcal P$ (general case). 
\end{definition} 

\begin{theorem} 
Sequences ${{T} \setminus \lambda \mathcal P}$ and $\mathcal T$ are realizations of double sieve.
\end{theorem} 

\begin{proof}
Sequences ${{T} \setminus \lambda \mathcal P}$ and $\mathcal T$ have the same structure in the sense of distribution of zero terms.
We consider the structure of sequence ${{T} \setminus \lambda \mathcal P}$:
\tabcolsep=0.092em
\begin{center}
\begin{tabular}{rrrrrrrrrrrrrrrrrrrrrrrrrrr}
${{\mathcal B}}$ = &{7,}&{13,}&{19,}&{0,} &{31,}&{37,}&{43,}&{0,} &{0,} &{61,}&{67,}&{73,}&{79,}&{0,} &{0,} &{97,}&{103,}&{109,}&{0,}  &{0,}&{127, \ldots} \\ 
${{\mathcal A}}$ = &{5,}&{11,}&{17,}&{23,}&{29,}&{0,} &{41,}&{47,}&{53,}&{59,}&{0,} &{71,}&{0,} &{83,}&{89,}&{0,} &{101,}&{107,}&{113,}&{0,}&{0, \ldots}   \\ 
\hline
${{T} \setminus \lambda \mathcal P}$ = &{2,}&{2,}&{2,}&{0,}&{2,}&{0,}&{2,}&{0,}&{0,}&{2,}&{0,}&{2,}&{0,}&{0,}&{0,}&{0,}&{2,}&{2,}&{0,}&{0,}&{0, \ldots}   
\end{tabular}
\end{center}
Elements of both sequences ${\mathcal A}$ and ${\mathcal B}$ labelled as $0$ mapped on ${{T} \setminus \lambda \mathcal P}$ as $0$ too.
Therefore the indices of zero terms of sequence ${{T} \setminus \lambda \mathcal P}$ will be determined by all expressions (\ref{a_a_composite}), (\ref{a_b_composite}), (\ref{b_a_composite}), and (\ref{b_b_composite}). 
Combining (\ref{a_a_composite}) with (\ref{b_a_composite}) and (\ref{a_b_composite}) with (\ref{b_b_composite}), we get two families of arithmetic progressions, $\lambda \in  \textbf{Z}^+$,
\begin{align} 
	 i =   \lambda a_j \pm j , \label{a_a_composite_2} \\ 
	 i =   \lambda b_k \pm k . \label{b_b_composite_2}  
	\end{align}
Equation (\ref{a_a_composite_2}) asserts that in sequence ${{T} \setminus \lambda \mathcal P}$ for each $p \in A$ there are two infinite disjoint subsequences of zero terms, the indices of which form arithmetic progressions with the same common difference $p \in A$.
Equation (\ref{b_b_composite_2}) asserts that in sequence ${{T} \setminus \lambda \mathcal P}$ for each $p \in B$ there are two infinite disjoint subsequences of zero terms, the indices of which form arithmetic progressions with the same common difference $p \in B$.
The proof follows from the statement $A \cup B \supset  \mathcal P$.
\end{proof}

Thus, the sum of the sequences $\mathcal A$ and $\mathcal B$, and also the sum of the sequences ${\mathcal L}$ and ${\mathcal R}$ are realizations of double sieve.
For all $p \in \mathcal P$ we have
\begin{align*}
\#^* ({T_m \setminus \lambda p}) \approx m \left( 1 - 2/p \right).	 
\end{align*}

\subsubsection{General case of double sieve}\label{general_sieving}

In a general case, that is, when we sum sequences derived from $\mathcal A$ and $\mathcal B$ or ${\mathcal L}$ and ${\mathcal R}$, 
we can get a sequence where two infinite subsequences of zero terms, which are arithmetic progressions with one and the same common difference may coincide 
and to degenerate into one subsequence. 
All the binary constructions that were considered in \ref{primes_with_fixed_gap}, \ref{twinprimes}, and \ref{reprofevennumber}
have a general view $S'_{m} \pm S''_{m}={S}_{m} = \{ \mathrm g \}_{i=1}^{i=m}$ and are generated by the identities $s'_i \pm s''_i = s_i =\mathrm g$. 
If $p \mid  \mathrm g$, then $p \mid s'_i$ if and only if $p \mid s''_i$. 
In this case both elements $s'_i$ Рё $s''_i$ maps to single element $s_i$.
Let $p, q \in \mathcal{P}$ and let $p \mid  \mathrm g$ and $q \nmid  \mathrm g$; then
\begin{align}
\#^* ({S_m \setminus \lambda p}) \approx  m \left( 1 - 1/p \right), \label{one_sieving} \\ 
\#^* ({S_m \setminus \lambda q}) \approx  m \left( 1 - 2/q \right). \label{two_sieving} 
	\end{align}
	
We give two examples to illustrate these assertions. 
Let $S\{ 28 \} = A^5 - B $. 
The difference between corresponding elements of these sequences is equal to $\mathrm g=28$. 
The even number $\mathrm g=28$ have the only one prime factor $p=7$, $p \in \mathcal{P}$.
Sieving this construction by a prime $p=7$, we get
\tabcolsep=0.08em
\begin{center}
\begin{tabular}{rrrrrrrrrrrrrrrrrrrrrrrrrrr}
${A}^5 \setminus \lambda 7 =$ &{0,}&{41,}&{47,}&{53,}&{59,}&{65,} &{71,}&{0,}&{83,}&{89,}&{95,}&{101,}&{107,}&{113,}&{0,}&{125,}&{131,}&{137,}&{143, \ldots} \\ 
${B}\setminus \lambda 7 =$  &{7,}&{13,}&{19,}&{25,}&{31,}&{37,}&{43,}&{0,}&{55,}&{61,}&{67,}&{73,}&{79,}&{85,}&{0,}&{97,}&{103,}&{109,}&{115, \ldots}  \\ 
\hline 
${S \{ 28 \}} \setminus \lambda 7 =$ &{0,}&{28,}&{28,}&{28,}&{28,}&{28,}&{28,}&{0,}&{28,}&{28,}&{28,}&{28,}&{28,}&{28,}&{0,}&{28,}&{28,}&{28,}&{28, \ldots}   
\end{tabular}
\end{center}
where in compliance with (\ref{one_sieving}) from sequence $S\{ 28 \} \setminus \lambda 7 $ sieved out one from each seven terms. 
Sieving this construction by a prime $p=5$, $5 \nmid 28$ gives us
\tabcolsep=0.078em
\begin{center}
\begin{tabular}{rrrrrrrrrrrrrrrrrrrr}
${A}^5 \setminus \lambda 5 =$ &{0,}&{41,}&{47,}&{53,}&{59,}&{0,}&{71,}&{77,}&{83,}&{89,}&{0,}&{101,}&{107,}&{113,}&{119,}&{0,}&{131,}&{137,}&{143, \ldots} \\ 
${B} \setminus \lambda 5 =$ & {7,}&{13,}&{19,}&{0,}&{31,}&{37,}&{43,}&{49,}&{0,}&{61,}&{67,}&{73,}&{79,}&{0,}&{91,}&{97,}&{103,}&{109,}&{0, \ldots} \\ 
\hline 
${S\{ 28 \}} \setminus \lambda 5=$ &{0,}&{28,}&{28,}&{0,}&{28,}&{0,}&{28,}&{28,}&{0,}&{28,}&{0,}&{28,}&{28,}&{0,}&{28,}&{0,}&{28,}&{28,}&{0, \ldots}   
\end{tabular}
\end{center} 
where in compliance with (\ref{two_sieving}) from sequence ${S\{ 28 \}} \setminus \lambda 5$ sieved out two from each five terms. 
Such a result we get for any $p \nmid 28$.

Construction for double sieving can be described roughly as 
  \begin{align} \label{twinprime}
   {\pi}_{2} (n)  \sim  m  \prod \limits_{5 \leqslant p \leqslant n} \left( {1 - \frac{2}{p}} \right), 
 \end{align}
where $m=n/6$. 
For all $\mathrm g = 2^{k}$ we have
\begin{align*} 
   {\pi}_{\mathrm g} (n)  \sim  {\pi}_{2} (n) 
	                   \prod \limits_{ n \leqslant p \leqslant n + \mathrm g } \left( {1 - \frac{1}{p}} \right) 
										 \leqslant {\pi}_{2} (n), 
 \end{align*}
as $n \to \infty$. 
For any even $\mathrm g$ we will have 
\begin{align*} 
   {\pi}_{\mathrm g} (n)  \sim \kappa \, \eta_2 (\mathrm g) \, {\pi}_{2} (n) 
	                   \prod \limits_{ n \leqslant p \leqslant n + \mathrm g } \left( {1 - \frac{1}{p}} \right),
										 \end{align*}
where
\begin{align*}  
   \eta_2 (\mathrm g) =  \prod_{\substack { p \mid \mathrm g}, p \in  \mathcal P}   \left(\frac{p - 1} {p-2} \right),
\end{align*}
and $\kappa = 1$ if $3 \nmid {\mathrm g}$ and $\kappa = 2$ if $3 \mid{\mathrm g}$ in compliance with the remark \ref{rem_on_diff} in section \ref{primes_with_fixed_gap}.

\section{Additive problems }

\subsection{Twin prime conjecture}\label{twinprimeconj}

\begin{proposition} \label{prop_1}
There exists the function $H_m$ such that: 
\begin{enumerate}
	\item  ${\pi}_2 (6m)>mH(m)$ for all sufficiently large $m$;
	\item $mH(m) \to \infty $ as $m \to~\infty$.
\end{enumerate}
\end{proposition}
\begin{proof}
Combining the inequalities (Rosser, Schoenfeld \cite{Rosser-Schoenfeld}) 
$$ \pi (x) > \frac {x} {\log x} \left( 1 + \frac{1} {2 \log x} \right), \  x \geqslant 59 $$  
and 
$$ \frac{{e^{ - \gamma } }} {{\log x}} \left( 1 + \frac{1} {2 (\log x)^2} \right) > \prod\limits_{p \leqslant x} {\left( {1 - \frac{1} {p}} \right)}, \ x > 1,$$
we obtain 
\begin{align*} 
  \pi (x) > x {e^{ \gamma }} \prod\limits_{ p \leqslant x} {\left( {1 - \frac{1} {p}} \right)}.  
\end{align*} 
This inequality holds for all sufficiently large $x$. Then it is easy to check that 
\begin{align*} 
n {e^{ \gamma }} \prod\limits_{ p \leqslant n} {\left( {1 - \frac{1} {p}} \right)} <  \pi (a,n) + \pi (b,n), \ n \geqslant 31;  
 \end{align*} 
\begin{align} \label{single_sieve1}
  m {e^{ \gamma }} \prod\limits_{5 \leqslant p \leqslant 6m} {\left( {1 - \frac{1} {p}} \right)}  < \pi (a, 6m) ,\ m \geqslant 2; 
\end{align} %
\begin{align} \label{single_sieve2}
  m {e^{ \gamma }} \prod\limits_{5 \leqslant p \leqslant 6m} {\left( {1 - \frac{1} {p}} \right)}  < \pi (b, 6m),\ m \geqslant 9. 
\end{align} %
Assume that there exist functions $m \varphi_{a,m}   = \pi (a,6m)$ and $ m \varphi_{b,m}   =  \pi (b,6m).$
Then from inequalities (\ref{single_sieve1}) and (\ref{single_sieve2}) imply that for all sufficiently large $m$
\begin{equation*} 
{e^{ \gamma }} \prod\limits_{5 \leqslant p \leqslant 6m} {\left( {1 - \frac{1} {p}} \right)} <  \frac{\pi (a,6m)}{m} = \varphi_{a,m} , \   
{e^{ \gamma }} \prod\limits_{5 \leqslant p \leqslant 6m} {\left( {1 - \frac{1} {p}} \right)} <  \frac{\pi (b,6m)}{m} =\varphi_{b,m}.  
\end{equation*}	
By the Dirichlet's theorem on primes in arithmetic progressions we get $\varphi_{a,m} \sim  \varphi_{b,m} $ as $m \to \infty$.
To proceed to double sieve we use relation 
\begin{align} \label{repeated_twin}
{C_{\left\langle 1,2 \right\rangle} (m)} &= {  \prod \limits_{5 \leqslant p \leqslant 6m}} \left( {1 - \frac{2}{p}} \right) \Big{/}
{  \prod \limits_{5 \leqslant p \leqslant 6m}} \left( {1 - \frac{1}{p}} \right)^2  \\
	&= { \prod \limits_{5 \leqslant p \leqslant 6m}} \frac{p(p-2)}{(p-1)^2} \nonumber \\ \nonumber 
	&= \frac {4} {3} { \prod \limits_{3 \leqslant p \leqslant 6m}} \frac{p(p-2)}{(p-1)^2}.	 \nonumber 
	\end{align}

Then from (\ref{repeated_twin}) and considering the estimates (\ref{single_sieve1}) and (\ref{single_sieve2}), we get 
\begin{align}\label{H2}
H(m) &= { {e^{2 \gamma }} \prod \limits_{5 \leqslant p \leqslant 6m}} \left( {1 - \frac{2}{p}} \right)  \nonumber\\
     &= {  {C_{\left\langle 1,2 \right\rangle} (m) }  {e^{2 \gamma }}  \prod \limits_{5 \leqslant p \leqslant 6m}} \left( {1 - \frac{1}{p}} \right)^2.
\end{align} 
The inequality ${\pi}_2 (6m) > m H_m$ holds for all $m>5$ and satisfy to the first statement of Proposition~\ref{prop_1}. 
Proof of the second statement of this Proposition can be obtained from the asymptotic equality (Rosser, Schoenfeld, \cite{Rosser-Schoenfeld})
\begin{align*}
   \prod \limits_{\alpha < p \leqslant x} \left( 1 - \frac {\alpha} {p} \right)  \cong \frac {c(\alpha)} {(\log x)^\alpha},
\end{align*}
where $\alpha$ is a real constant, and known limit
\begin{align*}
\lim_{x \to + \infty} \frac {x^{\beta}}{(\log{x})^{\alpha}}  = \infty
\end{align*}
which exists for all ${\alpha} >0$ and ${\beta} >0$.
\end{proof}

\subsubsection{The first Hardy-Littlewood conjecture} \label{first_HL_conj}

We may get the approximation of the number of twin primes $\pi'_2(6m)$ through the number of primes in the following way: 
since $m \varphi_{a,m}= \pi (a,6m)$ and $m \varphi_{b,m}= \pi (b,6m)$ 
we can rewrite the equation (\ref{H2}) as 
\begin{align} \label{myfirstHLconj}
\pi'_2(6m) = C_{\left\langle 1,2 \right\rangle}  (m) \cdot m  \frac {\pi (a,6m) } {m} \cdot \frac { \pi (b,6m)} {m}.    
\end{align}
Having substituted $m=n/6$, ${\pi (a, 6m) \cdot \pi (b, 6m)} \sim [\pi(n) /2 ]^2 $, and $C_{\left\langle 1,2 \right\rangle} = 4C_2 /3$ (where $C_2=0.6601\ldots$ is the twin prime constant) in (\ref{myfirstHLconj}), we obtain 
\begin{align*} 
  \pi'_2(6m) = 
          2 n C_2 \left[ \frac {\pi(n)} {n} \right]^2 
          \sim  2 C_2  \frac {n} {(\log n)^2}. 
  \end{align*}
Therefore, (\ref{myfirstHLconj}) is equivalent to the first Hardy-Littlewood conjecture.  
This is a better approximation of $\pi_2(6m)$ than $mH_m$. Thus, ${\pi}_2 (10^6) - \pi'_2(10^6) = 32.5356\ldots$  
It is known that $\pi_2(n)$ increases much like the function 
\begin{align*} 
2 C_2 \int^n_2  \frac {dt} {(\log t)^2}.
\end{align*}
It expresses $\pi_2(n)$ through the density of primes.  
The density of prime numbers in summands segments here decreases from $1/ \log 2$ to $1/ \log n$. 
Some transformation of it allow us to write the expression for the number of pairs of primes not exceeding $n + \mathrm g$ which have the difference ${\mathrm g}$:
\begin{align}  \label{firstHLconj}
  \frac {{\pi}_{{\mathrm g}}(n)}{\kappa \eta_2 (\mathrm g)} 
		\sim 2 C_2 \int^n_2  \frac {dt} {\log t \log (t + \mathrm g) }.
\end{align}
Here we sum two segments where the density of prime numbers decreases from $1/ \log 2$ to $1/ \log n$ in one segment 
and decreases from $1/ \log (2 + \mathrm g)$ to $1/ \log (n + \mathrm g)$ in the other.
We are not trying to find a good approximation for some quantities, but to reveal the asymptotic behavior of certain functions.

\subsection{The Goldbach-Euler conjecture}\label{Eul_Gol_conj}

\begin{proposition} \label{Eul_Gol_prop}
There exists the function $H'_m$ such that: 
\begin{enumerate}
	\item for all sufficiently large $m$ the number of representations as a sum of two primes each of the three successive even numbers $\mathrm g_m^1$, $\mathrm g_m^2$, and $\mathrm g_m^3$ are no less than $mH'_m$;
	\item $mH'_m \to \infty $ as $m \to \infty $.
\end{enumerate}
\end{proposition}%

\begin{proof}
The construction for $G(\mathrm g ; p)$ is sum of two segments, where 
the density of prime numbers decreases $1/ \log 2$ to $1/ \log n$ in one segment and increases from $1/ \log n$ to $1/ \log 2$ in the other segment.
To estimate the number of representations of an even number ${\mathrm g}$ as the sum of two primes we suppose that $n= \mathrm g$ and make a suitable choice of $\eta_2 (\mathrm g)$ and $\kappa$, by analogy with (\ref{firstHLconj}), we have   
\begin{align*} 
  \frac {G(\mathrm g ; p)} {\kappa \eta_2 (\mathrm g)}  \sim 2 C_2 \int^{n-2}_2  \frac {dt} {\log t \log(n-t)}.
\end{align*}
We may to take an estimate of $G(\mathrm g ; p)$ from ${\pi}_{{\mathrm g}}(n)$ through the quotient: 
\begin{align*} 
  \mu_2 (n) &=   {\sum^{n-2}_2  \frac {1} {\log t \log(n-t)}} \Big{/}  {\sum^{n-2}_2  \frac {1} {( \log t)^2}} \\
  &\approx    {\int^{n-2}_2  \frac {dt} {\log t \log(n-t)}}  \Big{/} {\int^{n-2}_2  \frac {dt} {( \log t)^2}}. 
\end{align*}
The quotient $\mu_2 (n)$ has its minimum $0.706\ldots$ at $n=32$ and attains the value $0.972\ldots$ at $n=10^5$. 
Thus, function $H'_m = \mu_2 (6m) H_m$ satisfy Proposition \ref{Eul_Gol_prop}. 
It is easy to see that  
$G(\mathrm g ; p)  \sim   \mu_2 (\mathrm g)  \kappa \eta_2 (\mathrm g)  {\pi}_2 {(\mathrm g)} \to~\infty$ as $\mathrm g \to~\infty $. 
\end{proof}

\section{Ordered set of prime numbers (tuples) } \label{tuple}

Let $(p_1, p_2, \ldots p_k)$ be the set of primes. Such a set of primes (tuple) is determined by the pattern $\omega_k =(g_1=0,g_2, \ldots g_k)$ where $g_i = p_i - p_1$. 
Three cases are possible:
\begin{enumerate}
	\item $g_2,g_3, \ldots g_k \equiv 0,2 \pmod 6 \Rightarrow g_1 \in A$;
	\item $g_2,g_3, \ldots g_k \equiv 0,4 \pmod 6 \Rightarrow g_1 \in B$;
	\item $g_2 \equiv g_3 \equiv \ldots  \equiv  g_k \equiv 0 \pmod 6 \Rightarrow g_1,g_2,g_3,\ldots g_k  \in A$ or $g_1,g_2,g_3,\ldots g_k  \in B$.
\end{enumerate}
On the other hand, the pattern $ \omega_k = (g_1, g_2, \ldots g_k)$ can satisfy other sets (tuples) of prime numbers $(p_r, p_r+g_2, \ldots p_r+g_k)$, where $\{p_r,r=1, 2,\ldots\}$ is subset of primes.
To find tuples corresponding to pattern $\omega_k$ it is necessary to sum $k$ sequences of the type ${\mathcal A}^{m'}$ or ${\mathcal B}^{m'}$.
These sequences must correspond to elements of pattern $\omega_k = (g_1=0,g_2,g_3,\ldots g_k)$: 
\begin{itemize}
	\item if $p_1 \in A$, then $m' = g/6$ for $p \in A$ and $m' = (g-2)/6$ for $p \in B$; 
	\item if $p_1 \in B$, then $m' = (g+2)/6$ for $p \in A$ and $m' = g/6$ for $p \in B$.
\end{itemize}
The sum of $k$ sequences contains $k - \nu_p$ infinite subsequences $\{0\}_p$ for every $p \in \mathcal P$.
The value of $\nu_p$ determined by the set $g_1,g_2, \ldots g_k$.
Let 
\begin{align*} 
   g_{11} \equiv g_{12} &\equiv \ldots \equiv g_{1j_1} \!\!\!\!\!\! \pmod{p}; \, \\
	 g_{21} \equiv g_{22} &\equiv \ldots \equiv g_{2j_2} \!\!\!\!\!\! \pmod{p}; \,  \\
	 \ldots  \, \\
	 g_{i1} \equiv g_{i2} &\equiv \ldots \equiv g_{ij_i} \!\!\!\!\!\! \pmod{p}, 
\end{align*}
and all $g_{11} , g_{21}  \ldots  g_{i1}$ are pairwise incongruent modulo $p$. Then $\nu_p =  -i + \sum_i {j_i}$. 
Construction for the number of $k$-tuples for which $p_r \leqslant n$ can be described roughly as follows: 
\begin{align*}   
  {\pi} (\omega_k, n) &\sim m \prod \limits_{p \geqslant 5}  \left( {1- \frac{k - \nu_p}{p}} \right)  \nonumber \\
	&= m \prod \limits_{5 \leqslant p \leqslant k }  \left( {1- \frac{k - \nu_p}{p}} \right)  \prod \limits_{p > k }  \left( {1- \frac{k - \nu_p}{p}} \right) \nonumber \\
	&= m \prod \limits_{5 \leqslant p \leqslant k }  \left( {1- \frac{k - \nu_p}{p}} \right)  \prod\limits_{{p > k }} \left( {1 + \frac{\nu_p}{p -k}} \right) 
	          \prod\limits_{{p > k }}   \left( {1- \frac{k}{p}} \right).  \nonumber 
	\end{align*}
Necessary condition for the existence of $k$-tuples is the inequality $k - \nu_p < p $ for every $p \in \mathcal P$.
In this case $C(\omega_k) = \prod_{5 \leqslant p \leqslant k }  \left( {1- ({k - \nu_p}) / {p}} \right)  >0$ and $\eta (\omega_k) =\prod_{{p > k }} \left( {1 + {\nu_p}/({p -k})} \right) \geqslant 1$.
To proceed to $k$-fold sieve we use relation 
\begin{align*} 
{C_{\left\langle 1,k \right\rangle} {(m)}} = {  \prod \limits_{k+1 \leqslant p \leqslant 6m}} \left( {1 - \frac{k}{p}} \right) \Big{/}
{  \prod \limits_{k+1 \leqslant p \leqslant 6m}} \left( {1 - \frac{1}{p}} \right)^k  
	\end{align*}
or 
\begin{align*}
{ {e^{k \gamma }} \prod \limits_{k+1 \leqslant p \leqslant 6m}} \left( {1 - \frac{k}{p}} \right)  
     = {  {C_{\left\langle 1,k \right\rangle} {(m)} }  {e^{k \gamma }}  \prod \limits_{k+1 \leqslant p \leqslant 6m}} \left( {1 - \frac{1}{p}} \right)^k.
\end{align*} 	
By analogy with (\ref{H2}) and in view of the estimate  (\ref{single_sieve1}) and (\ref{single_sieve2}), we get 
\begin{align} \label{k-fold sieve 2}
 {\pi} (\omega_k, n)  \sim m C(\omega_k ) \kappa \eta (\omega_k ) 
      {  {C_{\left\langle 1,k \right\rangle} {(m)} }  {e^{k \gamma }}  \prod \limits_{k+1 \leqslant p \leqslant 6m}} \left( {1 - \frac{1}{p}} \right)^k.
\end{align} 
If every $g_i \equiv 0 \pmod 6$, then $\kappa =2$. 
In all other cases $\kappa =1$. 
Since 
$$\lim_{m \to + \infty} m \prod \limits_{k+1 \leqslant p \leqslant 6m} \left( {1- \frac{k}{p}} \right) =
 \lim_{m \to + \infty} m {C_{\left\langle 1,k \right\rangle} {(m)}} {  \prod \limits_{k+1 \leqslant p \leqslant 6m}} \left( {1 - \frac{1}{p}} \right)^k = \infty,$$
and other factor in (\ref{k-fold sieve 2}) have a finite value, then $\lim_{n \to + \infty} \pi ({\omega_k}, n) = \infty$. 
That is if $k$-tuple is admissible then there exist infinitely many such $k$-tuples.

\subsection{Patterns of two pairs of twin primes} \label{patterns}

Let ${\Pi}_{m'}(n)$ be the number of twins $p, p+2$ not exceeding $n$ such that $p' = p +6m', p'+2 $ are also twins.
These $4$-tuples determined by pattern $\omega_4 = (0,2,6m',6m'+2)$.
We have the following construction for the solution of this problem 
\begin{align*} 
{\Pi}_{m'} (6m)  =  \#^*   ({\mathcal A_m} + {\mathcal B_m} + {\mathcal A^{m'}_m + {\mathcal B^{m'}_m} })  =  \#^*   ({\mathcal T_m} + {\mathcal T^{m'}_m } ). 
\end{align*} 
Suppose that $m' =1$. Then we obtain the number of prime quadruplets. 
We will show Proposition \ref{patterns1}.
\begin{proposition} \label{patterns1}
There exists the function $Q_m$ such that: 
\begin{enumerate}
	\item the inequality ${\Pi}_1 (6m)> mQ_m$ holds for all sufficiently large $m$;
	\item $mQ_m \to \infty$ as $m \to \infty$.
\end{enumerate}
\end{proposition}
\begin{proof}
For transition from the single sieve to the fourfold sieve we can use relations  
 \begin{align} \label{Q_41}
  C_{\left\langle 1,4 \right\rangle} {(m)} &={\prod \limits_{5 \leqslant p \leqslant 6m}} \left( 1- \frac{4}{p}\right) \Big{/} 
	             {\prod \limits_{5 \leqslant p \leqslant 6m}} \left( 1- \frac{1}{p} \right)^4 \\
	&= { \prod \limits_{5 \leqslant p \leqslant 6m}}  \frac{(p-4)p^3}{(p-1)^4}. \nonumber
  \end{align} %
Then from (\ref{Q_41}) and in view of the estimate (\ref{single_sieve1}), (\ref{single_sieve2}), and (\ref{H2}), we get 
\begin{align*} 
  Q(m) &= { {e^{4 \gamma }} \prod \limits_{5 \leqslant p \leqslant 6m}} \left( {1 - \frac{4}{p}} \right) \\ 
	     &= C_{\left\langle 1,4 \right\rangle} {(m)} {e^{4 \gamma }}{\prod \limits_{5 \leqslant p \leqslant 6m}} \left( 1- \frac{1}{p} \right)^4 \\
	 \end{align*}  
The inequality ${\Pi}_1 (6m)> mQ(m)$ holds for all $m>1$. Some numerical data supporting it, for example, ${\Pi}_1 (6m)= 166$ and ${\Pi}_1 (6m)- mQ(m) \approx 52.07$ when $6m=10^6$.
A better approximation can be obtained by repeating the technique from section {\ref{first_HL_conj}}:
\begin{align} \label{qoud_prime} 
  \Pi'_1(6m) &= m C_{\left\langle 1,4 \right\rangle}(m) \left[ \frac {\pi (a, 6m) } {m} \cdot \frac { \pi (b, 6m)} {m} \right]^2,    
  \end{align}
where $ \lim_{m \to + \infty} C_{\left\langle 1,4 \right\rangle} (m) = 0,3074950\ldots$
For $6m=10^6$, we have $\Pi_1(6m) -\Pi'_1(6m) = 8.39\ldots$ 
Therefore we may assume  ${\Pi}_{m'} (6m) = \eta_4 (m') {\Pi}_{1} (6m)$.
In common case
${\Pi}_{m'} (6m) > \eta (\omega_4) \Pi'_1 (6m)$, where
\begin{align*}  
    \eta (\omega_4)  
		= 	\frac {{\Pi}_{m'} (6m) } {{\Pi}_{1} (6m)} 
		=  \prod_{\substack {p \mid m', p \in  \mathcal P}} \left( \frac{p - 2} {p-4} \right). 
		\end{align*} 
Number of prime quadruplets ${\Pi}_{1}(n)$ can be expressed by the density of primes: 
\begin{align*} 
  {\Pi}_{1}(n) &\sim  C_{\left\langle 1,4 \right\rangle} {(m)} \int^n_2  \frac {dt} {\log t \log (t+2) \log (t+6) \log (t+8)} \\
	             &\sim  C_{\left\langle 1,4 \right\rangle} {(m)} \int^n_2  \frac {dt} {(\log t)^4}.
\end{align*}
\end{proof}

\subsection{The strengthened Goldbach-Euler conjecture}\label{extGEconj}

Because the Goldbach--Euler conjecture and the twin primes conjecture hold true, we must answer the following question: how many even numbers are the sum of two primes, each of which has a twin? 
It is clear that if one of three consecutive even numbers $\mathrm g_m^1$, $\mathrm g_m^2$, $\mathrm g_m^3$ has such a representation, then the other two also have same representations. 
The sum of two arbitrary pairs of twin primes gives two representations for $\mathrm g_m^2$ and one representation for each $\mathrm g_m^1$ and $\mathrm g_m^3$ for some $m$.
We will consider this set of representations as a single representation.
Number of such representations we denote as $G(m; t \in \mathcal T)$. 
We checked $30,000$ triples of even numbers and found that only $12$ of them lack such representations. 
We therefore claim Proposition \ref{extGEconj1}.   

\begin{proposition}{\textnormal{(The strengthened Goldbach-Euler conjecture)}}. \label{extGEconj1}
All even numbers, except $ \mathrm g_m^1 $, $\mathrm g_m^2$, $\mathrm g_m^3$ for $m = 1, 16, 67, 86, 131,$ $ 151, 186, 191, 211, 226, 541, 701$, are the sum of two primes, each of which has a twin. 
Moreover, there exists the function $Q'_m$ such that:
\begin{enumerate}
	\item the inequality $G(m; t \in \mathcal T)  > m Q'_m $ holds for all sufficiently large $m$; 
	\item $m Q'_m \to \infty $ as $m \to \infty$.
\end{enumerate}
 \end{proposition}
\begin{proof}
The construction for the number of these representations is 
$G(m; t \in \mathcal T) = \#^* \left( {{\mathcal T}_{m} + {\mathcal T}'_{m} } \right).$  
We will repeating the technique from section {\ref{Eul_Gol_conj}}. A quotient 
\begin{align*} 
  \mu_4 (n) &=  \frac {\sum^{n-2}_2  \frac {1} {(\log t)^2 (\log(n-t))^2}}{\sum^{n-2}_2  \frac {1} {(\log t)^4}} \\
	&\approx \frac {\int^{n-2}_2  \frac {dt} {(\log t)^2 (\log(n-t))^2}}{\int^{n-2}_2  \frac {dt} {(\log t)^4 }} 
\end{align*}
has its minimum $0.136278\ldots$ at $n=227$ and attains the value $0.57533\ldots$ at $n=1.2\cdot10^5$. 
Thus, we obtain the relation $Q'_m = \mu_4 (6m) Q_m $.
The inequality $Q'_m >1$ holds for all $m \geqslant 947 \geqslant 701$, and therefore, every even number greater than $4208 = 6 \times 701 +2$ may be represented as the sum of two primes, each of which has a twin.
\end{proof}
Proposition \ref{extGEconj1} can be extended to a set of prime pairs with any even difference $d$.

\section{Densities of certain sequences} 

A natural density (or asymptotic density) $\delta \left( {S} \right)$ and the Schnirelman density $d \left( {S} \right)$ of the sequence $S_m$ are defined as 
\begin{align*}
 \delta \left( {S} \right)  =  \lim_{m \to \infty}  \frac { \left|S_{(m)}\right|} {m}, \, 
        d \left( {S} \right)  =  \liminf_{m \to \infty}  \frac { \left| S_{(m)} \right|} {m}. 
\end{align*}
We define the sum $S' \oplus S''$ of two sequences as the set of integers of the form $s'$, $s''$, or $s' + s''$.

Obviously, $\delta ( P ) = d ( P ) =0$, 
$\delta \left( \mathcal L \right) = \delta \left( \mathcal R \right) = \delta \left( \mathcal T \right)  = 0$ and
$d \left( \mathcal {L} \right) = d \left( \mathcal {R} \right) = d \left( \mathcal {T} \right) = 0$. 
It is easy to see that if the Goldbach-Euler conjecture is true, then $\delta \left( \emph {P} \oplus \emph {P} \right) > 0.5$.    
The method we present here for the proof of the Goldbach-Euler conjecture allows us to assert that  
$$ d \left( {{\mathcal L} \oplus {\mathcal L}} \right) = 
d \left( {{\mathcal L} \oplus {\mathcal R}} \right) = d \left( {{\mathcal R} \oplus {\mathcal R}} \right)=1, $$
that is, there exist sequences with density $0$ that are bases of order $2$. 

From the strengthened Goldbach-Euler conjecture it follows that  
$ \delta \left( {{\mathcal T} \oplus {\mathcal T}} \right) = 1 $,
that is, sequence ${\mathcal T}$ is the asymptotic basis of order $2$.

The set of primes $P$ with $1$ is the basis of order $3$. 

The set of pairs of twin primes is the asymptotic basis of order $3$.


\begin{thebibliography}{99}

\bibitem {Rosser-Schoenfeld} 
J. Barkley Rosser and Lowell Schoenfeld, (1962)  \textit{Approximate Formulas for Some Functions of Prime Numbers}, Illinois J. Math. 6, 64-94.

\end{thebibliography}
\end{document}